\documentclass{amsart}
\usepackage{amsmath}
\usepackage{amssymb}
\usepackage{amsthm}
\usepackage{mathrsfs}

\usepackage{graphicx}

\newtheorem{theorem}{Theorem}[section]
\newtheorem{lemma}[theorem]{Lemma}

\theoremstyle{definition}
\newtheorem{definition}[theorem]{Definition}
\newtheorem{example}[theorem]{Example}
\newtheorem{corollary}[theorem]{Corollary}

\theoremstyle{remark}
\newtheorem{remark}[theorem]{Remark}

\numberwithin{equation}{section}

\begin{document}

\title{On Rings of Baire one functions}

\author{A. Deb Ray}
\address{Department of Pure Mathematics, University of Calcutta, 35, Ballygunge Circular Road, Kolkata - 700019, INDIA} 
\email{debrayatasi@gmail.com}
\author{Atanu Mondal}
\thanks{The second author is supported by Council of Scientific and Industrial Research, HRDG, India. Sanction No.- 09/028(0998)/2017-EMR-1}

\address{Department of Pure Mathematics, University of Calcutta, 35, Ballygunge Circular Road, Kolkata - 700019, INDIA}
\email{mail.atanu12@yahoo.com}

\begin{abstract}
This paper introduces the ring of all real valued Baire one functions, denoted by $B_1(X)$ and also the ring of all real valued bounded Baire one functions, denoted by $B_1^*(X)$. Though the resemblance between $C(X)$ and $B_1(X)$ is the focal theme of this paper, it is observed that unlike $C(X)$ and $C^*(X)$ (real valued bounded continuous functions), $B_1^*(X)$ is a proper subclass of $B_1(X)$ in almost every non-trivial situation. Introducing $B_1$-embedding and $B_1^*$-embedding, several analogous results, especially, an analogue of Urysohn's extension theorem is established.
\end{abstract}

\keywords{$B_1(X)$, $B_1^*(X)$, zero set of a Baire one function, completely separated by $B_1(X)$, $B_1$-embedded, $B_1^*$-embedded.}
\subjclass[2010]{26A21, 54C30, 54C45, 54C50}
\maketitle

\section{Introduction}
\noindent The real valued Baire class one functions on a real variable is very important and intensively studied in Real analysis. Several characterizations of Baire one functions defined on metric spaces were obtained by different mathematicians in \cite{JPFEACAPR}, \cite{DZ} etc. 
Lee, Tang and Zhao \cite{LTZ} characterized Baire class one functions in terms of the usual $\epsilon$-$\delta$ formulation as in the case of continuous functions under the assumptions that X and Y are complete separable metric spaces. In \cite{LV}, Baire one functions on normal topological spaces have been investigated and characterized by observing that pull-backs of open sets are $F_\sigma$.\\\\
\noindent The study of rings of continuous functions was initiated long back and the theory was enriched by publication of several outstanding results, cited in \cite{GK}, \cite{GJ}, \cite{EHEWITT}, \cite{MHSTONE}, etc., established by various well known mathematicians around the first half of the twentieth century. The class of all real valued Baire one functions defined on any topological space is of course a superset of the class of all real valued continuous functions on the same space and therefore, it is a natural query whether one can extend the study on the class of Baire one functions. This paper puts forward some basic results which came out of such investigation.\\\\
In Section 2, we introduce two rings, the ring $B_1(X)$ of all the real valued Baire one functions and the ring $B_1^*(X)$ of all the real valued bounded Baire one functions. We observe that $B_1(X)$ is a lattice ordered commutative ring with unity and the other one, i.e., $B_1^*(X)$ is a commutative subring with unity (and a sublattice) of $B_1(X)$. The rest of this section is devoted to build the basics that are required to carry on research in this field.\\\\
\noindent Though the focal theme of this work is to observe the similarities of $B_1(X)$ with $C(X)$, we find in section 3 that the scenario is quite different. It is well known that $C(X) = C^*(X)$ defines a pseudocompact space $X$ and there are plenty of examples, including compact spaces, which are pseudocompact. But in case of $B_1(X)$ and $B_1^*(X)$, the equality occurs very rarely. \\\\
\noindent Section 4 introduces zero sets of $B_1(X)$ and discusses several algebraic equalities involving unions and intersections of  zero sets of Baire one functions. Also, the relationships between zero set of modulus (as well as, the power) of a Baire one function $f$ with the zero set of $f$ are established. The sets which are completely separated by $B_1(X)$ is characterized via zero sets of functions from $B_1(X)$. However, the well known result that two sets $A$ and $B$ are completely separated by $C(X)$ if and only if $\overline{A}$ and $\overline{B}$ are completely separated by $C(X)$ becomes one sided in the context of $B_1(X)$. That it is indeed one-sided is supported by an example.\\\\
\noindent The final section of this paper is devoted for developing the idea of $B_1$-embedding and $B_1^*$-embedding of a set and thereby obtaining an analogue of Urysohn's Extension theorem in this new context.

\section{$B_1(X)$ and $B_1^*(X)$}
\noindent Let $X$ be any arbitrary topological space and $C(X)$ be the collection of all real valued continuous functions from $X$ to $\mathbb{R}$. We define $B_1(X)$ as the collection of all pointwise limit functions of sequnces in $C(X)$.
So, $B_1(X)=\{f:X\rightarrow\mathbb{R}$ $ :$  $\exists  $  $\{f_n\} \subseteq C(X)$, for which $\{f_n(x)\}$ pointwise converges to $f(x)$ for all $x \in$ $X$ $\}$, is called the set of Baire class one functions or Baire one functions.
It is clear that, $C(X)\subseteq B_1(X)$.\\\\
Let $f$ and $g$ be two functions in $B_1(X)$. There exist two sequences of continuous functions $\{f_n\}$ and $\{g_n\}$ such that, $\{f_n\}$ converges pointwise to $f$ and $\{g_n\}$ converges pointwise to $g$ on $X$. Then
\begin{itemize}
	\item $\{f_n+g_n\}$ pointwise converges to $f+g$.
	\item  $\{-f_n\}$ pointwise converges to $-f$.
	\item  $\{f_n.g_n\}$  converges pointwise to $f.g$. 
	\item $\{|f_n|\}$ converges pointwise to $|f|$.
\end{itemize}
\noindent In view of the above observations, it is easy to see that $(B_1(X),+,.)$ is a \textbf{commutative ring with unity $\textbf{1}$} (where $\textbf{1}: X \rightarrow \mathbb{R}$ is given by $\textbf{1}(x)=1$, $\forall x \in X$) with respect to usual pointwise addition and multiplication. \\
\noindent In \cite{LV}, Baire one functions are described in terms of pull-backs of open sets imposing conditions on domain and co-domain of functions.   
 \begin{theorem}\cite{LV} \label{P1 thm_4.1} \textnormal{(i)} For any topological space $X$ and any metric space $Y$, $B_1(X,Y)$  $\subseteq \mathscr{F_\sigma}(X,Y)$, where $B_1(X,Y)$ denotes the collection of Baire one functions from $X$ to $Y$ and $\mathscr{F_\sigma}(X,Y)=\{f:X\rightarrow Y : f^{-1}(G)$ is an $F_\sigma$ set, for any open set $G \subseteq Y$\}.\\
 \textnormal{(ii)} For a normal topological space $X$, $B_1(X,\mathbb{R})$  $= \mathscr{F_\sigma}(X,\mathbb{R})$, where $B_1(X,\mathbb{R})$ denotes the collection of Baire one functions from $X$ to $\mathbb{R}$ and $\mathscr{F_\sigma}(X,\mathbb{R})=\{f:X\rightarrow \mathbb{R} : f^{-1}(G)$ is an $F_\sigma$ set, for any open set $G \subseteq \mathbb{R}$\}.
 \end{theorem}
\noindent Define a partial order `$\leq$' on $B_1(X)$ by $f \leq g$ iff $f(x) \leq g(x)$, $\forall x \in X$. It is clear that $(B_1(X), \leq)$ is a \textbf{lattice}, where $sup\{f,g\} = f\vee g= \frac{1}{2}\{(f+g)+|f-g|\}$ and $inf\{f,g\} = f\wedge g = \frac{1}{2}\{(f+g)-|f-g|\}$ both are in $B_1(X)$. Also for any $f,g,h \in B_1(X)$ 
\begin{itemize}
	\item $f \leq g \implies f+h \leq g+h$.
	\item $f \geq 0$ and $g \geq 0 \implies f.g \geq 0.$
\end{itemize} 
\noindent So, $B_1(X)$ is a \textbf{commutative lattice ordered ring with unity}. Moreover $C(X)$ is a \textbf{commutative subring with unity and also a sublattice of $B_1(X)$}.\\
If a Baire one function $f$ on $X$ is a unit in the ring $B_1(X)$ then $\{x \in X \ : \ f(x)=0\}= \emptyset$. The following result shows that in case of a normal space, this condition is also sufficient.
\begin{theorem}
For a normal space $X$, $f \in B_1(X)$ is a unit in $B_1(X)$ if and only if $Z(f) = \{x\in X \ : \ f(x) = 0\} =\emptyset$.
\end{theorem}
\begin{proof}
If $f\in B_1(X)$ is a unit then clearly the condition holds. Let $f \in B_1(X)$ be such that $Z(f) = \emptyset$. Define $\frac{1}{f}(x) = \frac{1}{f(x)}$, for all $x \in X$. To show that $\frac{1}{f} \in B_1(X)$. Let $U = (a, b)$ be any open interval in $\mathbb{R}$. It is enough to show that $\left(\frac{1}{f}\right)^{-1}(U)$ is $F_\sigma$. $\left(\frac{1}{f}\right)^{-1}(U) = \{x\in X \ : \ a < \frac{1}{f(x)} < b\}$. \\
Case 1 : Suppose $0 \notin U$. Then $\{x\in X \ : \ a < \frac{1}{f(x)} < b\} = \{x\in X \ : \ \frac{1}{b} < f(x) < \frac{1}{a}\}$ $= f^{-1}\left(\frac{1}{b}, \frac{1}{a}\right)$, when  $a \neq 0,b \neq 0$ and $\left(\frac{1}{f}\right)^{-1}(U) = f^{-1}\left(\frac{1}{b}, \infty\right) $ or $f^{-1}\left(-\infty, \frac{1}{a}\right)$ according as  $a=0$ or $b=0$ . In any case the resultant set is an $F_\sigma$ set. \\
Case 2 : Let $0 \in U$. Since $Z(f) = \emptyset$, $f(x) \neq 0$ and hence $\frac{1}{f(x)} \neq 0$, $\left(\frac{1}{f}\right)^{-1}(U) = \left(\frac{1}{f}\right)^{-1}\left((a, 0)\right) \cup \left(\frac{1}{f}\right)^{-1}\left((0, b)\right)$. Then $\left(\frac{1}{f}\right)^{-1}(U) = \{x\in X \ : \ a < \frac{1}{f(x)} < 0\} \bigcup \{x\in X \ : \ 0 < \frac{1}{f(x)} < b\} =$ $\{x\in X \ : \ -\infty < f(x) < \frac{1}{a}\} \bigcup \{x\in X \ : \ \frac{1}{b} < f(x) < \infty\}$ $= f^{-1}\left(-\infty, \frac{1}{a}\right) \cup f^{-1}\left(\frac{1}{b}, \infty\right)$, which is an $F_\sigma$ set.\\
Hence, $\frac{1}{f} \in B_1(X)$. 						   
\end{proof}
\noindent However, the following theorem provides an useful sufficient criterion to determine units of $B_1(X)$, where $X$ is any topological space.
\begin{theorem}\label{P1 thm_2.1}
Let $X$ be a topological space and $f\in B_1(X)$ be such that $f(x)>0$ \emph{(}or $f(x)<0$\emph{)}, $\forall x \in X$, then $\frac{1}{f}$ exists and belongs to $B_1(X)$.
\end{theorem}
\begin{proof}
Without loss of generality let, $f \in B_1(X)$ and $f(x)>0,  \forall x \in X$. Then there exists a sequence of continuous functions $\{f_n\}$ such that  $\{f_n\}$ converges pointwise to $f$. Now we construct a sequence of continuous functions $\{g_n\}$, where $g_n(x)=|f_n(x)|+\frac{1}{n}$, $\forall n \in \mathbb{N}$ and $\forall x \in X$. Clearly $g_n(x)>0$, $\forall n \in \mathbb{N}$ and $\forall x \in X$.\\
We now show that, $\{g_n(x)\}$ converges to $f(x)$, for each $x \in X$.\\ For each $x \in X$, $\{f_n(x)\}$ converges to $f(x)$ and $\{\frac{1}{n}\}$ converges to $0$ imply that $\{g_n(x)\}$ converges to $|f(x)|= f(x)$, $\forall x \in X$.\\
Consider the function $g: \mathbb{R}-\{0\} \rightarrow \mathbb{R}$ defined by $g(y)=\frac{1}{y}$, $g$ being continuous $\{g\circ g_n\}$ is a sequence of continuous functions from $X$ to $\mathbb{R}$. Our claim is $\{g \circ g_n\}$ converges to $g \circ f$ on $X$.\\
Let, $\epsilon > 0$. Then there exists a $\delta > 0$ such that $|(g\circ g_n)(x)-(g \circ f)(x)|$ $=|g(g_n(x))-g(f(x))| < \epsilon$,    whenever $|g_n(x)-f(x)|< \delta$  (By using continuity of $g$).\\
Since $\{g_n$\} converges pointwise to $f$, for $\delta >0$, $\exists$ $K \in \mathbb{N}$ such that, $|g_n(x)-f(x)|< \delta$, whenever $n \geq K$. So, $|(g\circ g_n)(x)-(g \circ f)(x)| < \epsilon$ whenever $n \geq K$.\\
Hence, $\{(g \circ g_n)\}$ converges pointwise to $g \circ f$, i.e. $g\circ f \in B_1(X)$. Now $(g\circ f)(x)=\frac{1}{f(x)}$ shows that $\frac{1}{f}$ belongs to $B_1(X)$.
\end{proof}
\noindent In the last theorem, we have shown that composition of a typical continuous function $g:\mathbb{R}-\{0\} \rightarrow \mathbb{R}$ given by $g(x)=\frac{1}{x}$ and a typical Baire one function $f:X\rightarrow \mathbb{R}$ produces a Baire one function $g\circ f$. In the next theorem we further generalize this.
\begin{theorem}\label{P1 thm_2.2}
Let $f$ be any Baire one function on $X$ and $g: \mathbb{R \rightarrow \mathbb{R}}$ be a continuous function then the composition function $g\circ f$ is also a Baire one function.
\end{theorem}
\begin{proof}
Since $f \in B_1(X)$, there exists a sequence of continuous functions $\{f_n\}$ which converges pointwise to $f$. The functions $g\circ f_n$ are all defined and continuous functions on $X$, $\forall n \in \mathbb{N}$. Let $\epsilon > 0$ be any arbitrary positive real number and $x \in X$. By continuity of $g$ there exists a positive $\delta$ depending on $\epsilon$ such that, $|(g\circ f_n)(x)-(g\circ f)(x)|=|g(f_n(x))-g(f(x))| < \epsilon$, whenever $|f_n(x)-f(x)|< \delta$. Again by using pointwise convergence of $\{f_n\}$ we can find a natural number $K$ such that, $|f_n(x)-f(x)|< \delta, \forall n \geq K$. So,  $|(g\circ f_n)(x)-(g\circ f)(x)|< \epsilon, \forall$ $n \geq K$.\\
Hence, $g\circ f$ is a Baire one function on $X$.
\end{proof}

\noindent We introduce another subcollection of $B_1(X)$, called bounded Baire one functions, denoted by $B_1^*(X)$ consisting of all real valued bounded Baire one functions on $X$. i.e., $B_1^*(X)= \{f\in B_1(X):f$ is bounded on $X\}$. $B_1^*(X)$ also forms a commutative lattice ordered ring with unity \textit{$1$}, which is a subring and sublattice of $B_1(X).$\\
\section{Is $B_1^*(X)$ always a proper subring of $B_1(X)$?}
\noindent In case of rings of continuous functions we have seen that there are spaces for which $C^*(X)$ coincides with $C(X)$, where $C(X)$ and $C^*(X)$ denote respectively the collection of all real valued continuous functions and the collection of all real valued bounded continuous functions on $X$. For example, if $X$ is compact then $C(X)=C^*(X)$. In fact, the spaces for which $C(X)=C^*(X)$ are known as pseudocompact spaces. But for Baire one functions, the scenario is quite different. We show in the next theorem that for most of the spaces unbounded Baire one functions do exist.
\begin{theorem} \label{P1 thm_3.1}
Let $X$ be any topological space. If $f \in C(X)$ is such that $0 \leq f(x) \leq 1$, $\forall x \in X$ and $0$ is a limit point of the range set $f(X)$, then there exists an unbounded Baire one function on $X$. (i.e., $B_1^*(X)$ is a proper subset of $B_1(X)$)
\end{theorem}
\begin{proof}
  For each $n \in \mathbb{N}$, define $g_n:X \rightarrow \mathbb{R}$ by\\
	$$g_n(x)= \begin{cases}
	n^2 f(x) & if \ x \in f^{-1}([0,\frac{1}{n}])\\
	\frac{1}{f(x)} & \ if \ x \in f^{-1}([\frac{1}{n},1])
			\end{cases}$$
Each $g_n$ is continuous and it is clear that $\{g_n(x)\}$ converges pointwise to the function $g:X \rightarrow \mathbb{R}$ defined by
	$$g(x)= \begin{cases}
				0 & \ if \ x \in Z(f)\\
	\frac{1}{f(x)} & \ if \ x \notin Z(f)
	\end{cases}$$
So, $g$ is a Baire one function on $X$. Since $0$ is a limit point of $f(X)$, $g$ is unbounded Baire one function.
\end{proof}
\begin{remark} \label{P1 cor_3.2}
If $B_1(X)=B_1^*(X)$, then for every $f \in C(X)$, $0$ cannot be a limit point of $f(X)$. In fact, we can say more, if $B_1(X)=B_1^*(X)$, then for every $f \in C(X)$, $r$ is not a limit point of $f(X)$, where $r$ is any real number. This follows from the fact that if $r$ is a limit point of the range set of the continuous function $f$, then $0$ becomes a limit point of the set $g(X)$, where $g=f-r$ and we can apply the previous theorem to the function $(0 \vee |g|) \wedge 1$ to get an unbounded function.
\end{remark}
\begin{remark}
One may observe that whenever $B_1(X)=B_1^*(X)$ and $X$ possesses atleast one non-constant continuous function then $X$ cannot be connected, because in that case $f(X)$ must be an interval and hence possesses a limit point. Also, it follows easily from $B_1(X)=B_1^*(X)$ that $X$ is pseudocompact. Therefore, it is natural to ask, Which class of spaces is precisely determined by $B_1(X)=B_1^*(X)$?
\end{remark}
\noindent As a consequence of Remark~\ref{P1 cor_3.2} we obtain,
\begin{theorem} \label{P1 thm_3.4}
A completely Hausdorff space $X$ \emph{(}i.e., where any two distinct points are completely separated by continuous function\emph{)} is totally disconnected if $B_1(X)=B_1^*(X)$.	
\end{theorem}
\begin{proof}
Since $X$ is completely Hausdorff, it possesses a non-constant continuous function. So, by previous remark $B_1(X)=B_1^*(X)$ implies $X$ is disconnected. Now assume that $C$ be any component of $X$ and $x,y$ be two distinct points in $C$. Since $X$ is completely Hausdorff, there exists a continuous function $f$ on $X$ such that $f(x)=0$ and $f(y)=1.$ Then the function $g=f\big|_C$ is a continuous function on $C$ and $C$ being connected $g(C)$ must be an interval. So, $f(X)$ has a limit point as $g$ is the restriction function of $f$ on $C$. This contradicts to the fact that $B_1(X)=B_1^*(X)$. So, $C$ cannot contain more than one point. Since $C$ is arbitrary component of $X$, every component of $X$ contains single point. Hence, $X$ is totally disconnected.
\end{proof}
\begin{remark}
Converse of the theorem is not true. $\mathbb{Q}$, the set of all rational numbers is both completely Hausdorff and totally disconnected but $\mathbb{Q}$ possesses an unbounded real valued continuous function, hence an unbounded Baire one function.
\end{remark}
\noindent In the context of ring homomorphism we may further get results similar to the known results about homomorphism from $C(Y)$ (or $C^*(Y)$) to $C(X)$. 
\begin{theorem} \label{P1 thm_3.6}
Every ring homomorphism $t$ from $B_1(Y)$ (or $B_1^*(Y)$) to $B_1(X)$ is a lattice homomorphism.
\end{theorem}
\begin{theorem} \label{P1 thm_3.7}
Every ring homomorphism $t$ from $B_1(Y)$ (or $B_1^*(Y)$) to $B_1(X)$ takes bounded functions to bounded functions.
\end{theorem}
\begin{corollary}
	If a completely Hausdorff space $X$ is not totally disconnected, then $B_1(X)$ cannot be a homomorphic image of $B_1^*(Y)$, for any $Y$.
\end{corollary}
\begin{corollary}
	$B_1(X)$ and $B_1^*(X)$ are isomorphic if and only if they are identical.
\end{corollary}
\begin{theorem}
	Let, $t$ be a homomorphism from $B_1(Y)$ into $B_1(X)$, such that $B_1^*(X) \subseteq t(B_1(Y))$. Then $t$ maps $B_1^*(Y)$ onto $B_1^*(X)$.
\end{theorem}

\section{Zero sets in $B_1(X)$}
\noindent The zero set $Z(f)$ of a function $f \in B_1(X)$ is defined by $Z(f)= \{x \in X: f(x)=0\}$ and the collection of all zero sets in $B_1(X)$ is denoted by $Z(B_1(X))$. We say a subset $E$ of $X$ is a zero set in $B_1(X)$ if $E=Z(f)$, for some $f \in B_1(X)$.\\
We call a set to be a cozero set in $B_1(X)$ if it is the complement of a zero set in $B_1(X)$. $Z(B_1(X))$ is closed under finite union and finite intersection as $Z(f) \cup Z(g)=Z(f.g)$ and $Z(f) \cap Z(g)=Z(f^2+g^2)=Z(|f|+|g|)$.\\
It is evident that, $Z(f)=Z(|f|)=Z(f^n)$, for all $f \in B_1(X)$ and for all $n \in$ $\mathbb{N}$, $Z(\textbf{0})=X$ and $Z(\textbf{1})=\emptyset$. Here $\textbf{0}$ and $\textbf{1}$ denote the constant functions whose values are $0$ and $1$ on $X$.\\\\
\textbf{Examples of zero and cozero sets in $B_1(X):$}
\begin{itemize}
	\item Every zero set of a continuous function is also a zero set of a Baire one function. Every clopen set $K$ of $X$ is in $Z(B_1(X))$, as it is a zero set of a continuous function.
	\item  For any $f \in B_1(X), \{x \in : f(x) \geq 0\}$ and $\{x \in : f(x) \leq 0\}$ belong to $Z(B_1(X))$ and they are the zero sets of the functions $f-|f|$  and $f+|f|$ respectively.
	\item For any real number $r\in \mathbb{R}$, $\{x \in X: f(x) \leq r \}$ and $\{x \in X: f(x) \geq r \}$ are also in $Z(B_1(X))$.
	\item  For any real number $r\in \mathbb{R}$, $\{x \in X: f(x) < r \}$ and $\{x \in X: f(x) > r \}$ are cozero sets in $B_1(X)$.\\
\end{itemize}  
\noindent It is easy to observe that, for any $f \in B_1(X)$ there exists $g \in B_1^*(X)$ given by $g=f \wedge \textbf{1}$ such that $Z(f)=Z(g)$. Hence $Z(B_1(X))$ and $Z(B_1^*(X))$ produce same family of zero sets. 
\begin{theorem} \label{P1 thm_4.2}
For any $f \in B_1(X), Z(f)$ is a $G_\delta$ set.
\end{theorem} 
\begin{proof}
Since $\mathbb{R}$ is a metric space, by Theorem~\ref{P1 thm_4.1}, $f^{-1}(G)$ is an $F_\sigma$ set, for every open set $G$ of $\mathbb{R}$. Therefore in particular $f^{-1}(\mathbb{R}-\{0\})$ is an $F_\sigma$ set. i.e. $Z(f)=f^{-1}(\{0\})$ is a $G_\delta$ set.
\end{proof}
\begin{corollary}
Every cozero set in $Z(B_1(X))$ is $F_\sigma$.
\end{corollary}
\noindent \textbf{Observation:} Countable union of zero sets in $B_1(X)$ need not be a zero set in $B_1(X)$. For example, $\mathbb{Q}$ can be written as a countable union of singleton sets and each singleton set is a zero set in $B_1(\mathbb{R})$, but $\mathbb{Q}$ is not a zero set in $B_1(\mathbb{R})$, as it is not a $G_\delta$ set in $\mathbb{R}$. However it can be proved that, $Z(B_1(X))$ is closed under countable intersection. To establish this result we need two important lemmas, which are already proved for the functions of a real variable. We generalize these results here for any arbitrary topological spaces.
  
\begin{lemma}\label{P1 lem_4.4}
	If $f:X\rightarrow \mathbb{R}$ is a bounded Baire one function with bound $M$, where $X$ is any topological space, then there exists a sequence of continuous function $\{f_n\}$ such that, each $f_n$ has the same bound $M$ and $\{f_n\}$ converges pointwise to $f$ on $X$.
\end{lemma}

\begin{proof}
	Let $\{g_n\}$ be a sequence of continuous functions that converges pointwise to $f$ on $X$. Suppose also that $|f(x)|\leq M$, $\forall x \in X$. Define $\{f_n\}$ by\\
	$f_n(x)=$
	$\begin{cases}
	-M & if$ $g_n(x) \leq -M\\
		g_n(x) & if$ $-M \leq g_n(x) \leq M\\
		M & if$ $g_n(x) \geq M.
	\end{cases}$\\
	Each $f_n$ of the sequence of functions $\{f_n(x)\}$ is continuous and has bound $M$. Also $\{f_n(x)\}$ converges pointwise to $f$
 on $X$ with bound $M$.
\end{proof}
\begin{lemma}\label{P1 lem_4.5}
	Let $\{f_k\}$ be a sequence of Baire one functions defined on a topological space $X$ and let $\sum \limits_{k=1}^{\infty}M_k < \infty$, where each $M_k >0$. If $|f_k(x)| \leq M_k$,  $\forall k \in \mathbb{N}$ and $\forall x \in X$, then the function $f(x)= \sum\limits_{k=1}^{\infty}f_k(x)$ is also a Baire one function on $X$.
\end{lemma}
\begin{proof}
	Since each $f_k$ is a Baire one function, for each positive integer $k$ there exists a sequence of continuous functions $\{g_{ki}\}_{i=1}^\infty$ on $X$ such that $\{g_{ki}\}_{i=1}^\infty$ converges pointwise to the function $f_k$ on $X$ and by lemma \ref{P1 lem_4.4} we can choose $\{g_{ki}\}_{i=1}^\infty$ in such a way that, $|g_{ki}| \leq M_k,  \forall i \in \mathbb{N}$.\\
	 For each $n\in \mathbb{N}$, let $h_n=g_{1n}+g_{2n}+ \ldots +g_{nn}$.\\
	 It is easy to verify that each $h_n$ is continuous on $X$. We will show that $\{h_n\}$ converges pointwise to $f$ on $X$.\\
	 Fix a point $x\in X$ and let $\epsilon >0$ be an arbitrary positive real number. Since  $\sum\limits_{k=1}^{\infty}M_k < \infty$, we can find $K \in \mathbb{N}$ so that $\sum\limits_{k=K+1}^{\infty}M_k < \epsilon$.\\Now choose an integer $N>K$ such that \\
	 $|g_{ki}(x)-f_k(x)|< \frac{\epsilon}{K}$ for $1\leq k \leq K$ and $\forall$ $i \geq N$.\\
	 For any $n \geq N$ we have,\\
	 $|h_n(x)-f(x)|=|\sum\limits_{k=1}^{n}g_{kn}(x)-\sum\limits_{k=1}^{\infty}f_k(x)|\leq |\sum\limits_{k=1}^{n}(g_{kn}(x)-f_k(x))| + |\sum\limits_{k=n+1}^{\infty}f_k(x)| $\\
	$\leq |\sum\limits_{k=1}^{K}(g_{kn}(x)-f_k(x))| + \sum\limits_{k=K+1}^{n}|g_{kn}(x)|+ \sum\limits_{k=K+1}^{\infty}|f_k(x)| < \sum\limits_{k=1}^{K}\frac{\epsilon}{K}+2\sum\limits_{k=K+1}^{\infty}M_k < 3\epsilon$.\\
	Since $\epsilon$ is arbitrary, it follows that $\{h_n(x)\}$ converges pointwise to $f$ on $X$. Hence $f(x)= \sum\limits_{k=1}^{\infty}f_k(x)$  is a Baire one function on $X$.
\end{proof}
\begin{theorem} \label{P1 thm_4.6}
$Z((B_1(X)))$ is closed under countable intersection.
\end{theorem}
\begin{proof}
Let $Z(f_n) \in Z(B_1(X)), \forall n\in \mathbb{N}$.\\
We define $g_n(x)=|f_n(x)| \wedge \frac{1}{2^n}$, $\forall x\in X$ and $\forall n \in \mathbb{N}$ and let $g(x)= \sum\limits_{n=1}^{\infty}g_n(x)$. Here, $|g_n(x)| \leq \frac{1}{2^n}, \forall n \in \mathbb{N}$ and $\forall x \in X$. Also $\sum\limits_{n=1}^{\infty}\frac{1}{2^n} < \infty$. So, by lemma \ref{P1 lem_4.5} $g \in B_1(X)$ and $Z(g)= \bigcap\limits_{n=1}^\infty Z(g_n)=\bigcap\limits_{n=1}^\infty Z(f_n)$.
\end{proof}
\begin{definition}
	Two subsets $A$ and $B$ are said to be completely separated in $X$ by $B_1(X)$, if there exists a function $f \in B_1^*(X)$ such that $f(A)=r$ and $f(B)=s$ with $r<s$ and $r\leq f \leq s$, $\forall x\in X$.
\end{definition}
\noindent It is enough to say that $A$ and $B$ are completely separated by $B_1(X)$, if we get a function $g$ in $B_1(X)$ satisfying $g(x) \leq r, \forall x \in A$ and $g(x) \geq s, \forall x \in B$, for then, $(r \vee g) \wedge s$ has the required property. Moreover we can replace $r$ and $s$ by $0$ and $1$, as there always exists a continuous bijection $h$ from $[r,s]$ to $[0,1]$ and $h \circ g$ is a Baire one function with the desired property.\\
\noindent It is well known that two sets $A$ and $B$ are completely separated by $C(X)$ if and only if $\overline{A}$ and $\overline{B}$ are completely separated by $C(X)$. But in case of completely separated by $B_1(X)$, this result is one-sided. Certainly, $\overline{A}$ and $\overline{B}$ are completely separated by $B_1(X)$ implies that $A$ and $B$ are completely separated by $B_1(X)$. But the converse is not true, as seen in the following example.
	\begin{example}
		In $[0,1]$, the sets $A=[0,1)$ and $B=\{1\}$ are completely separated by $B_1([0,1])$, because $\{f_n\} \subseteq C[0,1]$ defined by \\
	  	$f_n(x)=x^n$, $\forall x \in [0,1]$ and $\forall n \in \mathbb{N}$ converges to the function $f(x)$ defined by\\
		
			$f(x)=$
		$\begin{cases}
		0 & if$ $ 0 \leq x < 1\\
		1 & if$ $ x = 1.\\
		\end{cases}$\\
		So, $f$ belongs to $B_1^*[0,1]$ and $f(A)=0, f(B)=1$. But $\overline{A}$, $\overline{B}$ are not disjoint and therefore are not completely separated by $B_1([0,1])$.
	\end{example}
\begin{theorem} \label{P1 thm_4.9}
Two subsets of $X$ are completely separated by $B_1(X)$ if and only if they are contained in disjoint zero sets in $Z(B_1(X))$.
\end{theorem}
\begin{proof}
Let $Z(f)$ and $Z(g)$ are two members of $Z(B_1(X))$ such that $A \subseteq Z(f)$ and $B \subseteq Z(g)$ with $Z(f) \cap Z(g)= \emptyset$. Clearly, the zero set of  $|f|+|g|$  is empty and we may define $h(x)=\frac{|f(x)|}{|f(x)|+|g(x)|}$. Here $(|f(x)|+|g(x)|) > 0$ on $X$ and $|f|+|g| \in B_1(X)$, so by Theorem \ref{P1 thm_2.1} $\frac{1}{|f(x)|+|g(x)|} \in B_1(X)$ which implies $h(x)=\frac{|f(x)|}{|f(x)|+|g(x)|} \in B_1(X)$. Now, $h(Z(f))=0$ and $h(Z(g))=1$. So, $Z(f)$ and $Z(g)$ are completely separated by $B_1(X)$. Therefore $A$ and $B$ are completely separated by $B_1(X)$.\\
Conversely, let $A$ and $A'$ be completely separated by $B_1(X)$. So, there exists $f \in B_1(X)$ such that $f(A)={0}$ and $f(A')={1}$. The disjoint sets $F=\{x: f(x) \leq \frac{1}{3}\}$ and $F'=\{x: f(x) \geq \frac{2}{3}\}$ belong to $Z(B_1(X))$ and $A \subseteq F, A'\subseteq F'$.
\end{proof}

\section{$B_1$-embedded and $B_1^*$-embedded}
\noindent In this section we introduce the analogues of $C$-embedding and $C^*$-embedding, called $B_1$-embedding and $B_1^*$-embedding in connection with the extensions of Baire one functions.
\begin{definition}
A subset $Y$ of a topological space $X$ is called $B_1$-embedded in $X$, if each $f \in B_1(Y) $ has an entension to a $g \in B_1(X)$. i.e. $\exists$ $g \in B_1(X)$ such that $g\big|_Y=f$.\\
Similarly, $Y$ is called $B_1^*$-embedded in $X$, if each $f \in B_1^*(Y) $ has an entension to a $g \in B_1^*(X)$.
\end{definition}
\begin{theorem} \label{P1 thm_5.2}
Any $B_1$-embedded subset is $B_1^*$-embedded.
\end{theorem}
\begin{proof}
Let $Y$ be any $B_1$-embedded subset of $X$ and $f \in B_1^*(Y)$. There is a natural number $n$ such that $|f(y)| \leq n$, $\forall y \in Y$. Since $Y$ is $B_1$-embedded, it has an extension to $g \in B_1(X)$. Taking $h=(-n \vee g) \wedge n$. We get $h \in B_1^*(X)$ and $g(y)=h(y), \forall$ $y \in Y$. So, $h$ is an extension of $f$ and $h$ is bounded. Hence $Y$ is $B_1^*$-embedded in $X$.
\end{proof}
\noindent We establish an analogue of Urysohn's extension theorem for continuous functions which gives a necessary and sufficient condition for a subspace to be $B_1^*$-embedded in a topological space $X$. Before proving this theorem we need a lemma that ensures that the uniform limit of a sequence of Baire one functions is a Baire one function. It is to be noted that a special case of this particular result when functions are on real line has already been established. We prove the result in a more general setting.\\
\begin{lemma}\label{P1 lem_5.3}
Let $\{f_n\}$ be a sequence of Baire one functions on $X$. If $\{f_n\}$ converges uniformly to $f$ on $X$ then $f$ is a Baire one function on $X$. 
\end{lemma}
\begin{proof}
	Let $\{f_n\}$ be a sequence of Baire one functions converging uniformly to $f$ on $X$. By definition of uniform convergence, for each $k \in \mathbb{N}$, there exists a subsequence $\{f_{n_k}\}$ such that $|f_{n_k}(x)-f(x)| < \frac{1}{2^k}, \forall x\in X$. \\
	Consider the sequence $\{f_{n_{k+1}}-f_{n_k}\}$. Then\\
	$|f_{n_{k+1}}(x)-f_{n_k}(x)| \leq 	|f_{n_{k+1}}(x)-f(x)|+ 	|f_{n_{k}}(x)-f(x)| \leq \frac{1}{2^{k+1}}+\frac{1}{2^k}=(\frac{3}{2})2^{-k}$.\\
	Let $M_k=(\frac{3}{2})2^{-k}$ and note that,
	$|f_{n_{k+1}}(x)-f_{n_k}(x)| \leq M_k, \forall x\in X$ and $\sum\limits_{k=1}^{\infty}M_k < \infty$, where $M_k > 0, \forall k \in \mathbb{N}$. So,
	$\sum_{k=1}^{\infty}[f_{n_{k+1}}(x)-f_{n_k}(x)]$ is a convergent series and so by lemma \ref{P1 lem_4.5} the sum function $\sum_{k=1}^{\infty}[f_{n_{k+1}}(x)-f_{n_k}(x)]$ is a Baire one function on $X$. Now,
	$\sum_{k=1}^{\infty}[f_{n_{k+1}}(x)-f_{n_k}(x)]=\lim\limits_{N\to \infty}\sum_{k=1}^{N}[f_{n_{k+1}}(x)-f_{n_k}(x)]=f(x)-f_{n_1}(x)$. Since, $f_{n_1}$ is a Baire one function, $f$ is also a Baire one function.
\end{proof}
\begin{theorem}\label{P1 thm_5.4}
A subset $Y$ of a topological space $X$ is $B_1^*$-embedded in $X$ iff any pair of subsets of $Y$ which are completely separated in $Y$ by $B_1(Y)$ are also completely separated in $X$ by $B_1(X)$.
\end{theorem} 
\begin{proof}
Let $Y$ be a $B_1^*$-embedded subspace of $X$ and $P$, $Q$ be subsets of $Y$ completely separated in $Y$ by $B_1(Y)$. So, there exists $f \in B_1^*(Y)$ such that, $f(P)=0$ and $f(Q)=1$. Since $Y$ is $B_1^*$-embedded in $X$, $f$ has an extension to $g \in B_1^*(X)$ such that $g(P)=0$ and $g(Q)=1$. Hence, $P$, $Q$  are completely separated in $X$ by $B_1(X)$.\\
\noindent Conversely, assume that any two completely separated sets in $Y$ by $B_1(Y)$ are also completely separated in $X$ by $B_1(X$). Let $f_1 \in B_1^*(Y)$. Then there exists a natural number $m$ such that, $|f_1(y)| \leq m$, $\forall y \in Y$.\\
For $n \in \mathbb{N}$, let $r_n=(\frac{m}{2})({\frac{2}{3}})^n$. Then $3 r_{n+1}=2r_n, \forall n$ and $|f_1| \leq 3r_1$.\\
Let $A_1=\{y \in Y: f_1(y)\leq -r_1\}$ and $B_1=\{y \in Y: f_1(y)\geq r_1\}$. Clearly, $A_1$ and $B_1$ are disjoint members in $Z(B_1(Y))$ and therefore, are completely separated in $Y$ by $B_1(Y)$. By our hypothesis $A_1$, $B_1$ are completely separated in $X$ by $B_1(X)$. So, there exists $g_1\in B_1^*(X)$ such that, $g_1(A_1)=-r_1$ and $g_1(B_1)=r_1$ also $-r_1 \leq g_1 \leq r_1$.\\
 Let $f_2=f_1-g_1\big|_Y$, then $f_2 \in B_1^*(Y)$ and $|f_2|\leq 2r_1$ on $Y$.\\
 So, $|f_2|\leq 3r_2$.\\
 Let  $A_2=\{y \in Y: f_2(y)\leq -r_2\}$ and $B_2=\{y \in Y: f_2(y)\geq r_2\}$. Then $A_2$ and $B_2$ are completely separated in $X$ by member of $B_1(X)$. So, there exists $g_2:X \rightarrow [-r_2,r_2]$, such that, $g_2 \in B_1^*(X)$ and $g_2(A_2)=-r_2,$ $g_2(B_2)=r_2$. Taking $f_3=f_2-g_2\big|_Y$ we get $f_3\in B_1^*(Y)$ with $|f_3|\leq 2r_2=3r_3$.\\
 If we continue this process and use principle of mathematical induction then we get two sequences $\{f_n\} \subseteq B_1^*(Y)$ and $\{g_n\} \subseteq B_1^*(X)$, with the following properties:\\
 $|f_n|\leq 3r_n$ and $f_{n+1}=f_n-g_n\big |_Y, \forall n \in \mathbb{N}$.\\
 Let $g(x)=\sum\limits_{n=1}^{\infty}g_n(x), \forall x \in X$. It is clear that, $|g_n(x)|\leq r_n, \forall n$ and $\forall x \in \mathbb{X}$, $g_n \in B_1(X), \forall n \in \mathbb{N}$. Also, $\sum\limits_{n=1}^{\infty}r_n < \infty$, where $r_n > 0, \forall n \in \mathbb{N}$.\\
 So, by lemma \ref{P1 lem_4.5} $g$ is a Baire one function on $X$.\\
 We claim that, $g(y)=f_1(y),\forall y \in Y$.\\
 $g(y)=\lim\limits_{n \to \infty}[g_1(y)+g_2(y)+...+g_n(y)]=\lim\limits_{n \to \infty}[f_1(y)-f_2(y)+f_2(y)-f_3(y)+...+f_n(y)]=f_1(y)-\lim\limits_{n \to \infty}f_{n+1}(y)=f_1(y)$ (as  $|f_n|\leq 3r_n$ and $\sum\limits_{n=1}^{\infty}r_n$ is a geometric series with common ratio less than $1$). So, $f \in B_1^*(Y)$ has an extension $g \in B_1(X)$ and we know that if a Baire one function  $f \in B_1^*(Y)$ has an extension in $B_1(X)$, then it has an extension in $B_1^*(X)$. Hence, $Y$ is $B_1^*$-embedded.
\end{proof}
\begin{theorem}\label{P1 thm_5.5}
If $Y$ is any $B_1$-embedded subspace of a topological space $X$ then $Y$ is completely separated in $X$ by $B_1(X)$, from any zero set of $Z(B_1(X))$ disjoint from $Y$.
\end{theorem}
\begin{proof}
Let $Y$ be a $B_1$-embedded subspace of $X$ and $Z(f)$ be a zero set in $Z(B_1(X))$, disjoint from $Y$.\\
Let $h:Y \rightarrow \mathbb{R}$ be a map defined as, $h(y)= \frac{1}{|f(y)|}$. Here, $|f|\bigg |_Y \in B_1(Y)$ and $|f(y)| > 0$ on $Y$ as $Z(f) \cap Y = \emptyset$. So, by Theorem \ref{P1 thm_2.1}, $h \in B_1(Y)$. As $Y$ is $B_1$-embedded, $h$ has an extension $g \in B_1(X)$. Then $|f|.g \in B_1(X)$ and $(|f|.g)(Z(f))=0$, $(|f|.g)(y)=1, \forall$ $ y \in Y$. It follows that $Y$ and $Z(f)$ are completely separated in $X$ by $B_1(X)$.
\end{proof}
\begin{theorem} \label{P1 thm_5.6}
A $B_1^*$-embedded subset $Y$ of $X$ is a $B_1$-embedded subset in $X$ if and only if it is completely separated in $X$ by $B_1(X)$ from any disjoint zero set in $Z(B_1(X))$.
\end{theorem}
\begin{proof}
Let $B_1^*$-embedded subset $Y$ of $X$ is $B_1$-embedded then the result follows from Theorem \ref{P1 thm_5.5}.\\
Conversely, let $Y$ be completely separated in $X$ by $B_1(X)$ from any zero set in $Z(B_1(X))$ disjoint from $Y$. Let $f \in B_1(Y)$. We consider the function $tan^{-1}: \mathbb{R} \rightarrow (-\frac{\pi}{2},\frac{\pi}{2})$. By Theorem \ref{P1 thm_2.2} the function $tan^{-1}\circ f$ is a bounded Baire one function. As $Y$ is $B_1^*$-embedded in $X$, $tan^{-1}\circ f$ has an extension to a function $g \in B_1(X)$.\\
Let $Z=\{x\in X: |g(x)|\geq \frac{\pi}{2}\}$, then $Z \in Z(B_1(X))$ and $Z \cap Y = \emptyset$. So, by hypothesis, there exists a Baire one function $h:X \rightarrow [0,1]$ such that $h(Z)=0$ and $h(Y)=1$. We see that, $g.h \in B_1(X)$ and $\forall x \in X, |(g.h)(x)| < \frac{\pi}{2}$. So, $tan(g.h) \in B_1(X)$ and $\forall$ $y \in Y$, $tan(g.h)(y)=f(y)$. Hence, $Y$ is $B_1$-embedded.
\end{proof}
\begin{corollary}
For any topological space $X$, a zero set $Z \in Z(B_1(X))$ is $B_1^*$-embedded if and only if it is $B_1$-embedded.
\end{corollary}

\end{document}